\documentclass[11pt]{amsart}

\usepackage{amsthm}
\usepackage{amssymb}

\newtheorem{thm}{Theorem}[section]\newtheorem{lem}[thm]{Lemma}\theoremstyle{definition}\newtheorem*{defn}{Definition}
\newtheorem*{thmA}{Theorem A}\newtheorem*{CorB}{Corollary B}

\DeclareMathOperator*\Irr{Irr}

\begin{document}
\title[Clifford correspondences]{Clifford correspondences and irreducible restrictions of characters}
\author{Tom Wilde} 

\email{twilde10@gmail.com}
\begin{abstract}
For a finite group $G$ and complex character $\chi\in\Irr(G)$ that restricts irreducibly to a normal subgroup $N\vartriangleleft G,$ we prove a theorem about Clifford correspondences between the characters of subgroups of $G$ that induce $\chi.$ 
\end{abstract}
\maketitle

\section{Introduction}

In this note, all groups are finite. As usual, $\Irr(G)$ is the set of irreducible complex characters of the group $G.$ A character pair in $G$ is a pair $(A,\alpha),$ where $A$ is a subgroup of $G$ and $\alpha\in\Irr(A).$ (This notation is due to I. M. Isaacs \cite{Isaacs1995}.) If $\alpha^G=\chi$ for some $\chi\in\Irr(G),$ we say $(A,\alpha)$ induces $\chi.$ In this note, we present a result about Clifford correspondences between character pairs that induce a given $\chi\in\Irr(G).$ To state our result, it is convenient to refer to a graph $\mathcal C_N(\chi),$ defined as follows for each $\chi\in\Irr(G)$ and $N\vartriangleleft G.$

\begin{defn} For $\chi\in\Irr(G)$ and $N\vartriangleleft G,$ $\mathcal C_N(\chi)$ is the undirected graph whose vertices are the character pairs $(A,\alpha)$ with $\alpha^G=\chi,$ and whose edges are formed by joining each vertex $(A,\alpha)$ with each other vertex $(B,\beta)$ for which there exists a character pair $(C,\gamma)$ such that:
\begin{enumerate} 
\item $C\vartriangleleft A$ and $C\subseteq N;$  
\item $B$ is the stabilizer of $\gamma$ in $A;$
\item $\beta$ is the Clifford correspondent of $\alpha$ with respect to $\gamma.$
\end{enumerate}
\end{defn}

The graph $\mathcal C_N(\chi)$ is a generalisation of the Clifford induction graph $\mathcal C(\chi)$ introduced in \cite{Isaacs1995}, which corresponds to the case $N=G.$ 

$\mathcal C_N(\chi)$ has the same vertices as $\mathcal C_G(\chi)$ and its edges are a subset of those of $\mathcal C_G(\chi).$ For arbitrary $N\vartriangleleft G,$ $\mathcal C_N(\chi)$ may have fewer edges and therefore potentially, more connected components than $\mathcal C_G(\chi).$ Our object in this note is to prove the following theorem.

\begin{thmA} \label{mainthm}
Suppose $\chi\in\Irr(G)$ for a group $G.$ Let $N$ be a normal subgroup of $G.$ If $\chi$ restricts irreducibly to $N,$ then $\mathcal C_G(\chi)$ and $\mathcal C_N(\chi)$ have the same connected components.
\end{thmA}

In Theorem A, it is natural to look at the connected components that contain $(G,\chi).$ Informally, the theorem says that if $\chi_N$ is irreducible for some $N\vartriangleleft G,$ then any character pair that can be reached from $(G,\chi)$ by some chain of Clifford correspondences (each link going from a character to its Clifford correspondent, or vice versa), can still be reached using only Clifford correspondences with respect to characters of subgroups of $N.$ 

The proof of Theorem A is in the next section. We conclude this section with some remarks. First, we point out two known results that show that $\mathcal C_N(\chi)$ is actually connected in some important special cases:
\vspace{1mm}
\begin{enumerate}
\item I. M. Isaacs proved (Theorem 8.8 in \cite{Isaacs1995}) that if $\chi(1)$ is odd and $G$ is $p$ - solvable for all primes $p$ dividing $\chi(1),$ then $\mathcal C_G(\chi)$ is connected.
\vspace{1mm}
\item It follows from a theorem of E. C. Dade (Theorem 6.5 in \cite{Dade1985}) that if $\chi\in\Irr(G)$ restricts irreducibly to $N,$ where $N\vartriangleleft G$ with $|N|$ odd, then $\mathcal C_N(\chi)$ is connected. 
\end{enumerate}
\vspace{1mm}

Note that under the hypotheses of Theorem A alone, one cannot conclude that $\mathcal C_N(\chi)$ (or equivalently in view of the theorem, $\mathcal C_G(\chi)$) is connected. An example was given in \cite{Isaacs1995}, where $G=\mathrm{GL}_2(\mathbb F_3)$ and $\chi\in\Irr(G)$ is the character of degree $4.$ In this case, $\mathcal C_G(\chi)$ has two connected components. (This also appears as Example 6.4 in \cite{Dade1985}.)

In \cite{Isaacs1995}, connectedness of $\mathcal C_G(\chi)$ is an ingredient in the proof of an important theorem about monomial characters of $p$ - solvable groups. The oddness and solvability assumptions in $(1)$ above, are natural in that context, so Theorem A does not add anything. Rather, Theorem A seems to be worth recording as a basic property of the Clifford correspondence itself. Similarly, the following immediate corollary seems to be worth noting. It merely confirms intuition, but does not seem to have been reported before. Recall that a quasiprimitive character is one whose restriction to any normal subgroup is homogeneous.

\begin{CorB}
Let $\chi\in\Irr(G)$ and assume $\chi_N$ is irreducible where $N\vartriangleleft G,$ and further, that $\chi_M$ is homogeneous for all $M\vartriangleleft G$ with $M\subseteq N.$ Then $\chi$ is quasiprimitive.
\end{CorB}
\begin{proof}
The homogeneity condition means that the set containing $(G,\chi)$ only, is a connected component of $\mathcal C_N(\chi).$ As $\chi_N$ is irreducible, Theorem A shows that the same holds in $\mathcal C_G(\chi),$ which is equivalent to the conclusion.
\end{proof}

\section{Proof of Theorem A}
In this section we prove Theorem A. We need the following lemma, whose proof is due to I. M. Isaacs and is reproduced here by his kind permission.

\begin{lem}\label{lemC}
Suppose that $N$ and $R$ are normal subgroups of the group $G$ and set $L=N\cap R.$ Suppose $\chi\in\Irr(G)$ is such that $\chi_N$ is irreducible and $\chi_L$ is homogeneous. If $R/L$ is a perfect group, then $R=L.$
\end{lem}

\begin{proof}
Let $\theta\in\Irr(L)$ lie under $\chi.$ Then $\theta$ is $G$ - invariant by hypothesis. We will first prove the lemma under the additional assumption that $\theta$ is linear. 

We can in any case clearly assume that $\theta$ is faithful, so with $\theta$ assumed linear, $L$ is contained in the center of $G.$ Since $[N,R]\subseteq L,$ it follows that $[N,R,R]=1.$ By the three-subgroups lemma, $[R,R,N]=1.$ As $R/L$ is perfect, $R=[R,R]L,$ so since $L$ is central, $R$ centralizes $N.$ But $\chi_N$ is irreducible, so by Schur's lemma, $R$ acts by scalars in any representation affording $\chi.$ Since $\chi$ is faithful, $R$ is therefore central in $G.$ Since $R/L$ is perfect, $R=L.$ This proves the lemma when $\theta$ is linear.

To complete the proof, note that $(G,L,\theta)$ is a character triple in the sense of Chapter $11$ of \cite{Isaacs1976}. (That is, $N\vartriangleleft G$ and $\theta\in\Irr(L)$ is $G$ - invariant.) By Theorem 11.28 in \cite{Isaacs1976}, there exists a character triple $(\Gamma,A,\lambda)$ with $\lambda$ linear, and an isomorphism $(\tau,\sigma)$ from $(G,L,\theta)$ to $(\Gamma,A,\lambda),$ as described by Definition 11.23 of \cite{Isaacs1976}. In particular, $\tau$ is an isomorphism from $G/L$ to $\Gamma/A,$ and for $L\subseteq H\subseteq G,$ $H^\tau$ denotes the inverse image in $\Gamma$ of $\tau(H/L).$

Let $\Psi=\sigma_G(\chi)$ be the character corresponding to $\chi$ under the isomorphism $(\tau,\sigma),$ so $\Psi\in\Irr(\Gamma)$ lies over $\lambda.$ By properties $(b)$ and $(c)$ of the definition referred to above, $\Psi_{N^\tau}$ is irreducible. Also $N^\tau\cap R^\tau=L^\tau,$ and $R^\tau/L^\tau$ is perfect, being isomorphic to $R/L.$ Since $\lambda$ is linear, $R^\tau=L^\tau$ by the first part of the proof, so $R=L$ as required.
\end{proof}

\begin{proof}[Proof of Theorem A]
Assume for a contradiction that the group $G,$ its normal subgroup $N$ and its irreducible character $\chi\in\Irr(G)$ are a counterexample to Theorem A chosen to minimize firstly $|G|$ and secondly $|G:N|.$ 

It must be the case that $N<G,$ so we may choose $M\vartriangleleft G$ such that $M/N$ is a chief factor of $G.$ By choice of counterexample, $\mathcal C_M(\chi)$ and $\mathcal C_G(\chi)$ have the same connected components, so there exist vertices $(A,\alpha)$ and $(B,\beta)$ of $\mathcal C_M(\chi)$ that are not connected in $\mathcal C_N(\chi),$ but that are joined by an edge of $\mathcal C_M(\chi).$ By relabelling if necessary, we can assume that $B\subseteq A$ and $\beta^A=\alpha.$ 
Since $(\alpha^G)_N=\chi_N$ is irreducible, $\alpha_{A\cap N}$ is irreducible, so if $A<G,$ then by choice of counterexample, $\mathcal C_{A\cap N}(\alpha)$ and $\mathcal C_A(\alpha)$ have the same connected components, and so $(B,\beta)$ and $(A,\alpha)$ are connected in $\mathcal C_{A\cap N}(\alpha).$ However, it follows directly from the definitions of the two graphs, that any edge of $\mathcal C_{A\cap N}(\alpha)$ is also an edge of $\mathcal C_N(\chi).$ We then conclude that $(B,\beta)$ and $(A,\alpha)$ are connected in $\mathcal C_N(\chi),$ contrary to hypothesis. 

Hence $A=G$ and so $(A,\alpha)=(G,\chi).$ Thus $(B,\beta)$ is joined to $(G,\chi)$ by an edge in $\mathcal C_M(\chi).$ Hence, we can choose a character pair $(R,\varphi)$ that satisfies the following conditions:\vspace{1mm}
\begin{enumerate}
\item $R\vartriangleleft G$ and $R\subseteq M;$
\item $\varphi$ is a constituent of $\chi_R;$ 
\item $B$ is the stabilizer of $\varphi$ in $G;$ 
\item $\beta$ is the Clifford correspondent of $\chi$ with respect to $\varphi;$
\item $|R|$ is a large as possible subject to $(1)-(4).$
\end{enumerate}
\vspace{1mm}

Set $L=N\cap R$ and let $\theta$ be a constituent of $\chi_L.$ Let $S$ be the stabilizer of $\theta$ in $G.$ We claim that $S=G.$ Let $$\xi\in\Irr(S),\tau\in\Irr(B\cap S)\text{ and }\eta\in\Irr(R\cap S),$$ be the Clifford correspondents of $\chi,$ $\beta$ and $\varphi$ respectively, each with respect to $\theta.$ 
If $g\in B,$ then $\theta^g$ lies under $\varphi,$ and so $\theta^g$ is conjugate to $\theta$ in $R.$ Hence $B\subseteq RS$ and so $B=R(B\cap S).$ We know that $\beta_R=e\varphi$ for some positive integer $e,$ so $e\varphi=\beta_R=(\tau^B)_R=(\tau_{R\cap S})^R.$ Since each irreducible constituent of $\tau_{R\cap S}$ lies over $\theta,$ it follows by the uniqueness property of the Clifford correspondence that $\tau_{R\cap S}=e\eta.$ Also, $\tau^G=(\tau^B)^G=\chi.$ Hence $\tau^S$ lies over $\theta$ and induces $\chi,$ and so $\tau^S=\xi.$ It follows that $\tau$ is the Clifford correspondent of $\xi$ with respect to $\eta,$ so there is an edge between $(S,\xi)$ and $(B\cap S,\tau)$ in $\mathcal C_S(\xi).$ Now if $S<G$ then since $\xi_{N\cap S}$ is irreducible, it follows by our choice of counterexample that $(S,\xi)$ and $(B\cap S,\tau)$ are connected in $\mathcal C_{N\cap S}(\xi).$ But any edge of $\mathcal C_{N\cap S}(\xi)$ is also an edge of $\mathcal C_N(\chi),$ so $(S,\xi)$ and $(B\cap S,\tau)$ are connected in $\mathcal C_N(\chi).$ However, $(B\cap S,\tau)$ and $(B,\beta)$ are joined by an edge in $\mathcal C_N(\chi)$ by definition of $\tau.$ Thus $(B,\beta)$ and $(S,\xi)$ are connected in $\mathcal C_N(\chi).$ This is contrary to hypothesis, so $S=G$ as claimed, and so $\theta$ is $G$ - invariant and $\chi_L$ is homogeneous.

Since $R\nsubseteq N$ and $M/N$ is a chief factor, $NR=M.$ Then $R/L\cong RN/N=M/N$ is either abelian or perfect. If $R/L$ is perfect, then the conditions of Lemma \ref{lemC} are satisfied by $\chi,$ $N$ and $L,$ so $R=L.$ This is not the case, since $M>N.$ Hence $R/L$ is an abelian chief factor of $G.$

Now $\theta\in\Irr(L)$ is $G$ - invariant and $\varphi\in\Irr(R)$ lies over $\theta.$ Since $R/L$ is an abelian chief factor of $G,$ I. M. Isaacs' ``going up'' theorem (see Problem 6.12 in \cite{Isaacs1976}) shows that either $\theta^R$ is a multiple of $\varphi,$ or $\varphi_L=\theta.$ Since $\varphi$ is not $G$ - invariant, $\varphi_L=\theta.$ 

Since $\chi_N=(\beta^G)_N$ is irreducible, $G=NB.$ Also, $[N,R]\subseteq N\cap R=L$, so $N$ centralizes $R/L.$ Since $\theta$ is $G$-invariant, the characters $\varphi^g\in\Irr(R)$ for $g\in G,$ are constituents of $\theta^R=\varphi(1_L)^R=\sum_{\lambda\in\Irr(R/L)}\varphi\lambda.$ Since $\Irr(R/L)$ is centralized by $M=NR,$ it follows that $M\cap B=M\cap B^g$ for all $g\in G;$ thus $M\cap B\vartriangleleft G.$ 

Now $\beta_{M\cap B}$ is irreducible, since $(\beta^G)_M$ is irreducible. Let $T$ be the stabilizer of $\beta_{M\cap B}$ in $G.$ (This makes sense, since $M\cap B\vartriangleleft G.$) Clearly $B\subseteq T,$ so $T=T\cap MB=B(M\cap T).$ However $\chi_M=(\beta_{M\cap B})^M$ is irreducible, so $M\cap T=M\cap B,$ and so $T=B.$ Hence $\beta$ is the Clifford correspondent of $\chi$ with respect to $\beta_{M\cap B}.$ In particular, of conditions $(1)-(5)$ satisfied by $(R,\varphi),$ $(1)-(4)$ are also satisfied by $(M\cap B,\beta_{M\cap B}).$ However, $R\subseteq M\cap B,$ so since $(R,\varphi)$ also satisfies condition $(5),$ we must have $R=M\cap B.$ In particular, since $\chi_M=(\beta_{B\cap M})^M=(\beta_R)^M,$ we see that $\chi$ vanishes on $M\backslash R.$ 

We claim that $\chi$ also vanishes on $R\backslash L.$ The argument for this is related to the proof of Lemma 6.1 in \cite{Isaacs1994}. Define a subset $\Delta\subseteq\Irr(R/L)$ as follows: $$\Delta=\left\{\lambda\in\Irr(R/L):\varphi\lambda=\varphi^g\text{ for some }g\in N\right\}.$$ 

If $\lambda,\mu\in\Delta$ then $\varphi\lambda\mu=(\varphi^g)\mu=(\varphi\mu)^g=\varphi^{hg}$ for some $g,h\in N,$ so $\lambda\mu\in\Delta.$ Furthermore, if $\lambda\in\Delta$ and $g\in B,$ then $\varphi(\lambda^g)=(\varphi\lambda)^g=\varphi^{ng}=\varphi^{g^{-1}ng}$ for some $n\in N,$ so $\lambda^g\in\Delta.$ These relations show that $\Delta$ is a subgroup of $\Irr(R/L)$ that is invariant under the conjugation action of $B.$ Since $G=NB$ and $N$ centralizes $R/L,$ $\Delta$ is invariant under the conjugation action of $G.$  

Now $|\Delta|>1$ since $NB=G$ and $\varphi$ is not $G$ - invariant. As $R/L$ is a chief factor of $G,$ this forces $\Delta=\Irr(R/L).$ Equivalently, every irreducible constituent of $\theta^R=\sum_{\lambda\in\Irr(R/L)}\varphi\lambda$ is conjugate to $\varphi.$ 

On the other hand, as noted earlier, each conjugate of $\varphi$ in $G$ is contained in $\theta^R.$ Since the characters $\varphi\lambda$ for $\lambda\in\Irr(R/L)$ are pairwise distinct, we see that $\theta^R$ is the sum of all the distinct conjugates of $\varphi.$ It follows that $\chi_R$ is a multiple of $\theta^R,$ so $\chi$ vanishes on $R\backslash L$ as claimed. 

We have shown that $\chi$ vanishes on $M\backslash R$ and on $R\backslash L.$ Thus $\chi$ vanishes on $M\backslash L.$ This is impossible, since $L\subseteq N<M$ and $\chi_N$ is irreducible. The proof is complete.
\end{proof}

\end{document}